\documentclass[11pt,twoside,a4paper]{article}
\usepackage{color}
\usepackage{srcltx}
\usepackage{fancybox}
\usepackage{fancyhdr}
\usepackage{latexsym}
\usepackage{amsfonts}
\usepackage{amscd}
\usepackage{epsfig}
\usepackage[cp850]{inputenc}
\usepackage[english]{babel}
\usepackage{ifthen}
\usepackage{graphicx}
\usepackage{float}
\usepackage{amsthm}
\usepackage{amsmath}
\usepackage{amssymb}
\usepackage{mathrsfs}
\usepackage[mathscr]{eucal}
\usepackage{dsfont}
\usepackage{amstext}
\usepackage{syntonly}

\usepackage{tikz} 
\usepackage{color}
\usepackage{graphicx}

  \usepackage{enumerate}
\theoremstyle{plain}
\newtheorem{theorem}{Theorem}
\newtheorem{proposition}[theorem]{Proposition}
\newtheorem{lemma}[theorem]{Lemma}
\newtheorem{corollary}[theorem]{Corollary}

\theoremstyle{remark}
\newtheorem{remark}[theorem]{Remark}

\theoremstyle{definition}
\newtheorem{definition}[theorem]{Definition}

\newtheorem{example}[theorem]{Example}

\usepackage{color}

\usepackage{tikz}
\usepackage{color}
\usepackage{graphicx}

\usepackage{tikz}
\usepackage{color}
\usepackage{graphicx}

\usepackage{color}

\usepackage{enumerate}
\usepackage{amssymb}
\usepackage[english]{babel}
\usepackage{amstext}
\usepackage{graphicx}
 \usepackage{mathptmx}      
\usepackage{latexsym}
\begin{document}

\title{A topological study for the existence of lower-semicontinuous Richter-Peleg multi-utilities \thanks{Asier Estevan acknowledges financial support
 from the Ministry of Economy and Competitiveness of Spain under grants  MTM2015-63608-P and ECO2015-65031 as well as from the Basque Government under grant IT974-16. 
Armajac Ravent\ 'os acknowledges financial support
 from the Ministry of Economy and Competitiveness of Spain under grant ECO2015-65031.
}
}







\maketitle

 \author{G. BOSI \footnote{Dipartimento de Scienze Economiche, Aziendali, Mathematiche e Statistiche\\ Universit\`a degli Studi di Trieste. Piazzale Europa 1, I-34127. Trieste, Italia.
} }

 \author{A. ESTEVAN \footnote{Departamento de Matem\'aticas, Universidad P\'ublica de Navarra,
Campus Arrosad\'{\i}a\\ Pamplona, 31006, Spain\\ asier.mugertza@unavarra.es}}

 \author{Armajac Ravent\ 'os \footnote{Departamento de Matem\'aticas, Universidad P\'ublica de Navarra,
Campus Arrosad\'{\i}a\\ Pamplona, 31006, Spain.}}

\begin{abstract}

In the present paper we study necessary and sufficient conditions for the existence of a semicontinuous and finite Richter-Peleg multi-utility for a preorder. It is well know that, given a preorder on a topological space,  if there is a lower (upper) semicontinuous  Richter-Peleg multi-utility, then the topology of the space must be finer than the Upper (resp. Lower) topology. However, this condition does not guarantee the existence of a semicontinuous representation.

We search for finer topologies which are necessary for semicontinuity, as well as that they could guarantee the existence of a semicontinuous representation. As a result, we prove that Scott topology (that refines the Upper one) must be contained in the topology of the space in case there exists a finite lower semicontinuous  Richter-Peleg multi-utility. However, as it is shown, the existence of this representation cannot be guaranteed.

\end{abstract}

\textbf{keywords: }{Preorders, multi-utility theory, Richter-Peleg, semicontinuity }

\section{Introduction and motivation} \label{s1}

  In the present paper we study the existence of  {\em lower-semicontinuous Richter-Peleg multi-utilities} for preorders on topological spaces. The existence of Richter-Peleg multi-utilities has been recently studied by Alcantud et al. \cite{Alc} (see the introduction in this paper in order to find the motivations for adopting such a representation).
    
 It was already observed that this kind of representations not always exists for preorders endowed with the {\em Upper topology} $\tau_u$ (see Theorem 3.1 in Alcantud et al. \cite{Alc}). On the other hand, it is well known that the weak lower contour sets of the preorder have to be closed in the more general case when there exists a  lower-semicontinuous multi-utility (see e.g. Proposition 2.1 in Bosi and Herden \cite{BosHer1}). Negative conditions for the existence of a finite (Richter-Peleg) continuous multi-utility representations were presented in Alcantud et al. \cite{Alc} and Kaminski \cite{Kam}.
 
 The goal of this paper is to identify some other topologies related to the preorder, with respect to which it is possible to characterize the lower-semicontinuous Richter-Peleg multi-utility representability. Therefore, these topologies have to be finer than the {\em Upper topology}. Some of them are well known in other fields of mathematics, such us the {\em Scott topology} \cite{lattices} in computing, or the {\em Alexandrov topology} in Pure Mathematics \cite{arenas,tim}. Scott topologies have been used in order to characterize the functions between lattices (in particular, dcpo-s) that preserve suprema of directed sets  \cite{lat,lattices}. In  any case, the present paper study the more general case of preorders and their  finite lower-semicontinuous Richter-Peleg multi-utilities, so we do not assume the existence of  suprema and we search for a family of functions that fully characterize the order structure (i.e. a multi-utility instead of a single utility function).

 In this line, we prove that if there exists a finite lower-semicontinuous Richter-Peleg multi-utility for a given preorder, then the topology of the space refines the Scott topology. Thus, we achieve a significant necessary condition for the existence of the desired representation: from now, if we search for finite lower semicontinuous Richter-Peleg multi-utilities we should start from a topological space that refines the Scott topology, and not the Upper.

 Furthermore, we also present an example in order to show that this necessary condition is not sufficient for the general case. Hence, we continue in the study of the adequate topologies to guarantee the existence of the lower-semicontinuous Richter-Peleg multi-utility. For that, we prove that there always exists this kind of representation when the preorder is endowed with a topology that is finer that the Alexandrov  topology. Thus, we achieved a sufficient condition.

Througthout the paper we also focus on some other results related to the topic in order to interact with our present results. For example, several authors have work under the hypothesis in which any linear extension of the preorder is lower-semicontinuous. From a topological point of view, this is strongly related to the Alexandrov topology. Therefore, this kind of topologies cannot be considered strange at all. 

The structure of the paper goes as follows. Section 2 contains the notation and the preliminaries. Section 3 presents necessary conditions for the existence of a (finite) lower semicontinuous multi-utility representation of a preorder. Section 4 is devoted to the sufficient conditions for the existence of such representations of preorders.

 \section{Notation and preliminaries} \label{s2}

 From now on $X$ will stand for a nonempty set.

 \begin{definition} A \emph{preorder} $\precsim$ on $X$ is a binary relation on $X$ which is reflexive and transitive.
 An antisymmetric preorder is said to be an \emph{order} or a \emph{partial order}. A \emph{total preorder} \rm $\precsim$ on a set $X$ is a preorder such that if $x,y \in X$ then $[x \precsim y] \vee [y \precsim x]$. A total order is also called a \emph{linear order}\rm, and a totally ordered set $(X,\precsim)$ is also said to be a \emph{chain}\rm.

 If $\precsim$ is a preorder on $X$, then as usual we denote the associated \emph{asymmetric} relation by $\prec$ and the
associated \emph{equivalence} relation by $\sim$
and these are defined, respectively, by $[x \prec y \iff (x \precsim y) \wedge\neg(y \precsim x)]$ and $[x \sim y \iff (x\precsim y) \wedge (y \precsim x)]$. 
The asymmetric part of a linear order (respectively, of a total preorder) is said to be a \emph{strict linear order} \rm (respectively, a strict total preorder).
 \end{definition}

 Next Definition~\ref{lkeste}
  introduces the notion of representability for  preorders.

 \begin{definition} \label{lkeste} A total preorder $\precsim$ on $X$ is called \emph{representable} \rm if there is
 a
real-valued function $u\colon X\to \mathbb R$ that is order-preserving, so that, for every $x, y \in X$, it holds that $[x \precsim y \iff u(x) \leq u(y)]$. The map $u$ is said to be a \emph{utility function}.
 
In case of not necessarily total preorder,  a
real-valued function $u\colon X\to \mathbb{R}$  is said to be a \emph{Richter-Peleg representation} if it satisfies that $[x \precsim y \Rightarrow u(x) \leq u(y)]$ (i.e. $u$ is \emph{isotonic}) as well as $[x \prec y \Rightarrow u(x) < u(y)]$. In case of a total preorder, this definition coincides with the previous one.
 
 A (not necessarily total) preorder $\precsim$ on a set $X$   is said to have a {\em
multi-utility representation}  if  there exists a family $\mathcal{U}$ of isotonic real functions such that
for all points $x ,y \in X$  the equivalence
\begin{equation} \label{mult1} x \precsim y \Leftrightarrow  \forall u \in {\mathcal U} \,\,(u(x) \leq u(y))\end{equation}
 holds.

A particular case of the previous representation is the so called {\em Richter-Peleg multi-utility representation} (\cite{Min2}), which holds when all the functions of the family $\mathcal U$ in representation (\ref{mult1}) are {\em order-preserving} with respect to the preorder $\precsim$ (i.e., for all $u \in {\mathcal U}$, and $x,y \in X$, $x \prec y$ implies that $u(x) < u(y)$). It is well known that in this case the family $\mathcal U$ also represents the {\em strict part} $\prec$ of $\precsim$, in the sense that, for all $x,y \in X$, $x \prec y$ if and only if $u(x) < u(y)$ for all $u \in {\mathcal U}$.
 \end{definition}

 It is known that a multi-utility representation  exists for every not necessarily total preorder $\precsim$ on $X$ (see Evren and Ok Proposition~1 in \cite{EvO}). However, there are preorders that fails to be Richter-Peleg multi-utility representable (see \cite{Alc}, see also \cite{partial}).

\begin{definition} \label{grande} 
A total preorder $\precsim$ defined on $X$ is said to be \emph{perfectly separable} \rm if there exists a countable subset $D \subseteq X$ such that for every $x,y \in X$ with $x \prec y$ there exists $d \in D$ such that $x \precsim d \precsim y$.
 \end{definition}

Theorem~\ref{qwert} on representability for total preorders is well known \cite{BRME}.

 \begin{theorem}\label{qwert}
 A total preorder $\precsim$ on $X$ is representable if and only if it is perfectly separable.
\end{theorem}

 \begin{definition} Let $\prec$
denote an
asymmetric binary relation on $(X, \tau)$.   Given $a \in X$ the sets
$L_{\prec}(a) = \lbrace t \in X \ : \ t
\prec a \rbrace $ and $
U_{\prec}(a) = \lbrace t \in X \ :
 \ a \prec t \rbrace $ are called, respectively, the \emph{strict lower and upper contours}
 \rm of $a$ relative to $\prec$. We say that $\prec$ is  \emph{$\tau$-continuous} \rm (or just \emph{continuous})
if for each $a \in X$ the  sets $L_{\prec}(a)$ and $U_{\prec}(a)$
are $\tau$-open.

We will denote the \emph{order topology} generated by $\prec$ as $\tau_{\prec}$, and it is defined by means of the  subbasis provided by the lower and upper contour sets.

Let $\precsim$ denote a
reflexive binary relation on $(X, \tau)$.   Given $a \in X$ the sets
$L_{\precsim}(a) = \lbrace t \in X \ : \ t
\precsim a \rbrace $ and $
U_{\precsim}(a) = \lbrace t \in X \ :
 \ a \precsim t \rbrace $ are called, respectively, the \emph{weak lower and upper contours}
 \rm of $a$ relative to $\precsim$. We say that $\precsim$ is  \emph{ $\tau$-lower semicontinuous} \rm (\emph{$\tau$-upper semicontinuous})
if for each $a \in X$ the  sets $L_{\precsim}(a)$ (resp. $U_{\precsim}(a)$)
are $\tau$-closed.
\end{definition}

\begin{definition}
A preorder $\precsim$ on a set $X$ is said to be {\em near-complete} if every subset of $X$ consisting of mutually incomparable elements is finite.
\end{definition}

The following result was presented by Evren and Ok \cite[Theorem 3]{EvO}.

\begin{theorem} \label{Tok}
Let $X$ be a topological space with a countable basis.  If $\precsim$
is a near-complete upper (lower) semicontinuous preorder on $X$, then it has an upper (lower) semicontinuous
finite multi-utility representation.
\end{theorem}

The theorem above presents a sufficient condition for the existence of an upper (lower) semicontinuous finite multi-utility; however, there is not a similar result for the case of an upper (lower) Richter-Peleg multi-utility.

\begin{definition}
Let $\precsim$ be a preorder defined on $X$. The \emph{Upper topology} $\tau_U$ is obtained by choosing the closed sets to be the weak lower contour sets (as well as their finite unions and infinite intersections). 
\end{definition}

\begin{definition}
We say that $f\colon (X,\tau)\to \mathbb{R}$  is {\em lower semi-continuous} at $x_0$ if for every $\epsilon >0$ there exists a neighborhood $ U$ of $x_0$  such that $f(x) > f(x_{0})-\epsilon $  for all $x\in U$.
\end{definition}

\begin{remark}  It is known that  $f\colon (X,\tau)\to \mathbb{R}$ is lower semi-continuous at $x_0$ if and only if $f$ is continuous with respect to the Upper topology on the real line associated with the natural (total) order $\leq$ on $\mathbb{R}$ (i.e., $f\colon (X,\tau)\to (\mathbb{R}, \tau_u )$ is continuous). Equivalently, $f$ is lower semi-continuous at $x_0$ if $f^{-1}((-\infty, f(x_0)])$ is closed. This can be expressed too as
    $$ \liminf _{x\to x_{0}}f(x)\geq f(x_{0}). $$\end{remark}

 \begin{definition}

 Let $\precsim$ be a binary relation on $X$. A subset $G\subseteq X$ is said to be an \emph{up-set} if
 $\forall x,y\in X$,  $x\in G $ and $ x\precsim  y$ implies that $ y\in G$.

 Dually, a subset $G\subseteq X$ is said to be a \emph{down-set} if
 $\forall x,y\in X$,  $x\in G $ and $ y\precsim  x$ implies that $ y\in G$.
 \end{definition}

\begin{theorem} \label{rgt}  A total preorder $\precsim$ on a topological space is representable through a continuous utility function if and only if  $\precsim$ is perfectly separable and $\tau$-continuous.
  \end{theorem}

Theorem~\ref{rgt} on continuous representability is also well-known in this lite\-ra\-tu\-re \cite{de,Debr,BRME}. 

\begin{corollary}
A preorder $\precsim$ on $(X, \tau)$ is $\tau$-lower semicontinuous if and only if the topology $\tau$ is finer than the Upper topology $\tau_u$.
\end{corollary}


\begin{definition}
Let $(X,\precsim)$ be a preordered set. The \emph{Alexandrov's  topology} $\tau_A$ on $X$ is defined by choosing the open sets to be the up-sets:
 $$ \tau_A =\{ \,G\subseteq X:\forall x,y\in X\, (x\in G \land  x\precsim  y)\, \Rightarrow\, y\in G\,\}$$


The corresponding closed sets are the down-sets:
$$ \{\,S\subseteq X:\forall x,y\in X \, (x\in S\, \land \, y\precsim  x)\, \Rightarrow\, y\in S\,\}$$
\end{definition}

\begin{definition}\label{DScott}
Let $(X,\precsim)$ be an ordered set. The \emph{Scott  topology} $\tau_S$ on $X$ is defined by choosing the open sets to be the up-sets that satisfy the following condition (for any directed set $(x_i)_{i\in I}\subseteq X$):
$$  \sup (x_i)_{i\in I}=s \in U \Rightarrow (x_i)_{i\in I}\cap U\neq \emptyset  .$$

The definition is analogous in the case of preordered sets, taking into account that now the supremum is unique (in case it exists) except indifference (i.e. equivalence). Hence, equivalent elements are topologically indistinguishable (i.e. they share the same open neighborhoods, see \cite{subm}) in the Scott topology.
 \end{definition}

It is straigtforward to see that the Upper topology is contained in the Alexandrov topology. The following example shows that this inclusion may be strict.

\begin{example}\label{Ex1}
Let $X$ be the infinite union $\bigcup_{n\in \mathbb{N}} X_n$ where $X_n=[0,+\infty)$ for each $n\in \mathbb{N}$ (we denote $X_n$ by $[0,\infty)_n$ and by $x_n$ any element of $X_n$).
Now we define the preorder $\sqsubseteq $ on $X$ by $x\sqsubseteq y $ if and only if $x,y\in X_k$ and $x\leq y$, with $k\in \mathbb{N}$. Hence, $x $ and $y$ are incomparable for any $x\in X_m$ and $y\in X_n$, for any $n\neq m$.

On this preordered set, notice that the subset $A=[1,+\infty)_1\cup X_2\cup\cdots\cup X_n\cup\cdots$ (since it is an up-set) is open in the Alexandrov topology, whereas it fails to be open in the Upper topology (there is no open neighbourhood of the element $1_1$ contained in $A$). Notice too that the subset $B= \bigcup_{n\in \mathbb{N}} (0,+\infty)$, for example, is also open in the Alexandrov topology, whereas it fails to be open in the Upper topology.

The reason that makes $A$ fail to be open in the Upper topology is that not every up-set can be open, only those of the kind $X\setminus \bigcup_{i=1}^n L_\sqsubseteq (a^i) $  ($a^i\in X$) are open in the Upper topology. So, in particular, notice that $A_{|X_1}=[1,+\infty)$ is not open even if we are just working on the set $X_1=[0, +\infty)$ with the usual order $\leq$.

On the other hand, the reason that makes $B$ fail to be open in the Upper topology is that the arbitrary intersection of open sets fails to be open (this property is satisfied by the Alexandrov topology, but not by the Upper).  Thus, for any open set $U\in\tau_u$ there are an infinite number of bottom elements $0_k$.
\end{example}

\section{Necessary properties for semi-continuity}

It is  known \cite{Alc} {\color{red} CITE} that the $\tau$-lower semicontinuity of the preorder is not enough in order to warrant the existence of a (lower-semicontinuous) Richter-Peleg multi-utility.

 The following example shows that 
even dealing with a preorder which is Richter-Peleg multi-utility representable finitely, the corresponging finite semicontinuous multi-utility respresentantion fails to exist.

\begin{example}\label{ceupper}
Let $\mathbb{R}_1$, $\mathbb{R}_2$ be two copies of the real line $\mathbb{R}$, and consider the set $X=\mathbb{R}_1\cup\mathbb{R}_2$ 
endowed with the Upper topology associated to the  partial order $\precsim$ defined as follows

$$ x\precsim y \iff \{ x\leq y, \,  x,y\in \mathbb{R}_i, \, i=1,2  \quad\text{ or } \quad x\leq y, \,  x\in (-\infty, 0)_i, \, y\in [0,+\infty)_j, \, i\neq j.\}  $$

It is easy to check that this preorder can be represented by means of a finite lower semicontinuous multi-utility, for instance through the following two functions:

$$u(x)=  \left\{  \begin{array}{lcl}
x& ;&x \in \mathbb{R}_1 \cup [0, +\infty)_2\\
0& ;& x \in (-\infty, 0)_2
\end{array}\right.
\qquad
v(x)=  \left\{  \begin{array}{lcl}
x& ;&x \in \mathbb{R}_2 \cup [0, +\infty)_1\\
0& ;& x \in (-\infty, 0)_1
\end{array}\right.
$$

\begin{figure}[htbp]
\begin{center}
\begin{tikzpicture}[scale=0.8]
\draw[thick] (-1.5,0.2) node[anchor=east] {$\mathbb{R}_1$};
\draw[dashed] (-1.7,0) -- (-1,0);
\draw[dashed] (0.85,0) -- (3.85,0);
\draw[dashed] (0.85,-2) -- (3.85,-2);
    \draw[] (-1,0) -- (0.85,0);
    \draw[] (3.85,0) -- (5,0);
\draw[dashed] (5,0) -- (6,0);

\draw[thick] (0.9,0.25) node[anchor=east] {\small $0_1$};



\draw[thick] (4.25,0.25) node[anchor=east] {\small $0_1$};

\draw[thick] (5.25,0) node[anchor=east] { $\bullet$};
\draw[thick] (5.5,0.25) node[anchor=east] {\small $1_1$};
\draw[thick] (3.95,0) node[anchor=east] {\small $[$};
\draw[thick] (3.95,-2) node[anchor=east] {\small $[$};
\draw[thick] (4,0) node[anchor=east] { $\bullet$};

\draw[thick] (5.25,-2) node[anchor=east] { $\bullet$};
\draw[thick] (5.5,-1.75) node[anchor=east] {\small $1_2$};

\draw[dashed] (0.85,0) -- (3.85,-2);
\draw[dashed] (0.85,-2) -- (3.85,0);

\draw[thick] (-1.5,-1.8) node[anchor=east] {$\mathbb{R}_2$};
\draw[dashed] (-1.7,-2) -- (-1,-2);
    \draw[] (3.85,-2) -- (5,-2);
    \draw[] (-1,-2) -- (0.85,-2);
\draw[dashed] (5,-2) -- (6,-2);
\draw[thick] (0.9,-1.75) node[anchor=east] {\small $0_2$};
\draw[thick] (1,-2) node[anchor=east] {\small $)$};
\draw[thick] (1,0) node[anchor=east] {\small $)$};


\draw[thick] (4.25,-1.75) node[anchor=east] {\small $0_2$};
\draw[thick] (4,-2) node[anchor=east] { $\bullet$};
\end{tikzpicture}
\caption{Preorder defined on $\mathbb{R}\times \{1,2\}$.}
\label{figureCounterExample}
\end{center}
\end{figure}
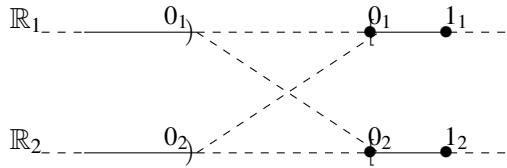

However, there is not a finite lower semicontinuous Richter-Peleg multi-utility for this preordered set and with respect to the Upper topology. To see that, first notice that for any $x_i\in (-\infty,0)_i$, the sequence $(-\frac{1}{n})_{n\in \mathbb{N}}$ contained in $ (-\infty,0)_j$ converges to $x_i$, as well as $-\frac{1}{n}\bowtie x_i$ and $-\frac{1}{n}\bowtie x_i+\epsilon$, for some $\epsilon>0$ and  $i\neq j$.
Hence, for any $n\in\mathbb{N}$, there is a function $u$ of the multi-utility such that $u(x_i)<u(x_i+\epsilon)<u(-\frac{1}{n})$. Furthermore, since the amount of functions is finite, it actually holds that there is a function $u$ and an infinite subset $M\subseteq \mathbb{N}$ such that $u(x_i)<u(x_i+\epsilon)<u(-\frac{1}{n})$ for any $n\in M\subseteq \mathbb{N}$. Thus, $u(x_i)<\liminf u(\frac{1}{n})$, so $u$ fails to be lower semicontinuous.
\end{example}
\medskip

From the examples above, we are able to extract the following conditions that must be satisfied for the existence of the desired representation.

\begin{proposition}

Let $\precsim$ be a preorder on a topological space $(X,\tau)$. Assume that there exists a lower-semicontinuous Richter-Peleg multi-utility. Let $(x_i)_{i \in I}$ be a 
net in $X$.
If $(x_i)_{i \in I}$ converges to $a$ and there is $b$ such that $b\bowtie x_i$ ($\forall i>i_0$), then $\neg (b\prec a)$ or the multi-utility is infinite.


\end{proposition}
\begin{proof}

By reduction to the absurd, if there is $b\in X$  such that $b\prec a$, then $u(b)<u(a)$ is satisfied for any function $u$ of the multi-utility.
On one hand,
if $(x_i)_{i \in I}$ converges to  $a$  and $u$ is a lower semicontinuous function, then it holds that $\liminf u(x_i)\geq u(a)$. On the other hand, if $x_i \bowtie b$ for any $i\in I$ then, for each $i\in I$ there must be two functions $u_i$ and $v_i$ in the multi-utility such that $ u_i(x_i)<u_i(b)$ as well as $ v_i(x_i)>v_i(b)$.

Thus, if the amount of functions is finite, then there is a subnet $(x_j)_{j\in J\subseteq I}$ and two functions $u$ and $v$ such that  $ u(x_j)<u(b)$ and $ v(x_j)>v(b)$, as well as $u(b)<u(a)$ and $v(b)<v(a)$. Hence, $ \liminf u(x_i)\leq u(b)< u(a)$, so $u$ fails to be lower semicontinuous at $a$, arriving to the desired contradicition.

\end{proof}

\medskip

Since it is necessary to ask for some properties to the Upper topology in order to achieve a lower-semicontinuous Richter-Peleg multi-utility, we decide to study some other topologies (finer than the Upper). Due to this deliberation, we achieve the following result.

\begin{theorem}
Let $\precsim$ be a preorder on a topological space $(X,\tau)$.  If there exists a finite lower-semicontinuous Richter-Peleg multi-utility, then $\tau$ is finer than the Scott topology $\tau_{Scott}$.  However,  this latter condition is not sufficient in order to guarantee the existence of a finite lower-semicontinuous Richter-Peleg multi-utility.
\end{theorem}
\begin{proof}
Let's see that any open set $U$ in the Scott topology is also open in $\tau$. That is,
let's see that  any up-set $U$ satisfying that $ \{ \sup (x_i)_{i \in I}=s \in U \Rightarrow (x_i)_{i \in I}\cap U\neq \emptyset \} $  (for any directed set $(x_i)_{i\in I}$) is contained in $\tau$. To see that, we shall prove that $U$ is an open neighbourhood of any of its points.

Let $x$ be any point of $U$. Since each function $u_k$ of the multi-utility $\mathcal{U}=\{u_k\}_{k=1}^N$ is lower semicontinuous, then for any $\epsilon> 0$ there exists an open neighbourhood $V_k^{\epsilon}$ of $x$ such that $u_k(V_k)\subseteq (u_{k}(x)-\epsilon, +\infty)$.
Now, we define the open set $V^\epsilon=\bigcap_{k=1}^N   u_k^{-1}((u_{k}(x)-\epsilon, +\infty))$.
Notice that for any $y\in \bigcap_{k=1}^N   u_k^{-1}([u_{k}(x), +\infty))$ it holds that $x\precsim y$. Dually,  for any $y\in \bigcap_{k=1}^N   u_k^{-1}((-\infty, u_{k}(x)])$ it holds that $y\precsim x$.
\medskip

We distinguish two cases:

\begin{itemize}
\item[$(i)$] If there is one $\epsilon_0>0$ such that $V^{\epsilon_0}\subseteq U$, then we conclude that $U$ is an open neighbourhood of $x$, finishing our proof.
\item[$(ii)$] If case $(i) $ does not hold, then for any $\epsilon>0$ it holds that $V^{\epsilon} \nsubseteq U$. Hence, for each $\epsilon=\frac{1}{n}$ ($n\in \mathbb{N}$) we can construct an increasing sequence $(x_n)_{n\in M\subseteq \mathbb{N}}$ such that each $x_n$ is in $ \bigcap_{k=1}^N u_k^{-1}((u_{k}(x)-\frac{1}{n}, +\infty))\setminus  u_k^{-1}((u_{k}(x)-\frac{1}{n+1}, +\infty))$. Notice that $x_n\prec x$ for any $n\in \mathbb{N}$, so $x$ is an upper bound of the sequence. Observe too that $\sup (u_k(x_n))_{n\in \mathbb{N}}= u_k(x).$

Now, we distinguish the following cases:

\begin{itemize}
\item[$(a)$] If $\sup (x_n)_{n\in \mathbb{N}}=\bar{x}\in U$, then we arrive to the absurd $(x_n)_{n\in\mathbb{N}}\cap U\neq \emptyset.$ That is, this case cannot hold.
\item[$(b)$] If  $\sup (x_n)_{n\in \mathbb{N}}=\bar{b}$, then $b\precsim x$. If $b\prec x$, then $u_k(b)<u_k(x)$ for any $k=1,...,n$. Thus, there is an $\epsilon_0=\min \{u_k(x)-u_k(b)\}_{k=1}^N >0$ such that $u_k(x_n)<u_k(x)-\epsilon_0$ for any $n\in \mathbb{N}$, which is absurd. That is, there is no element $b$ such that $x_n\prec b\prec x$.
If $b\sim x$, then $b\in U$ as in case (a) (remember that, according to Definition~\ref{DScott}, in that case $b$ and $x$ are indistinguishable, so they share the open neighborhoods).
\item[$(c)$]  If  $\sup (x_n)_{n\in \mathbb{N}}$ does not exist, then (since there is no element $b$ such that $x_n\prec b\prec x$) there must be  element $b$ such that $x_n\prec b$ for any $n\in \mathbb{N}$
as well as $b\bowtie x$ (otherwise $x$ would be the supremum and that would be the aforementioned case $(a)$). So, there is a function $u_j\in \mathcal{U}$ such that $u_j(b)<u_j(x)$. Hence, there is an $\epsilon_0= \{u_j(x)-u_j(b)\}_{k=1}^n >0$ such that $u_j(x_n)<u_j(x)-\epsilon_0$ for any $n\in \mathbb{N}$, so $ \liminf u_j(x_n)<u_j(x) $. Hence,  $u_j$ fails to be lower semicontinuous, arriving to a contradiction.
\end{itemize}
\end{itemize}

To conclude the proof, we show in the following example that, even if the topology $\tau$ is the Scott topology, that does not guarantee the existence of a finite lower semicontinuous Richter-Peleg multi-utility.
\end{proof}


\begin{example}\label{cescott}
Let $X=\{ (-\infty, 0)\cup \{1\}\cup  [2, +\infty]\}$ be a set endowed with the Scott topology associated to the  partial order $\precsim$ defined as follows:

$$ x\precsim y \iff x\leq y, \, \forall y\in X\setminus\{1\},  \, \forall x\in X, \quad\text{ and } \quad 1\bowtie y, \,\, \forall y<0. $$

Let's see that there is no  finite lower semicontinuous  Richter-Peleg multi-utility for this preordered set.

Let $\mathcal{U}$ be a finite Richter-Peleg multi-utility.
First, notice that the sequence $(-\frac{1}{n})_{n\in\mathbb{N}}$ converges to 2, as well as $1\prec 2$ and $1\bowtie -\frac{1}{n}$ for any $n\in\mathbb{N}$. Hence, for any $n\in\mathbb{N}$, there is a function $u_n$ of the multi-utility such that $u(-\frac{1}{n})<u(1)<u(2)$. Furthermore, since the amount of functions is finite, it actually holds that there is a function $u$ and an infinite subset $J\subseteq \mathbb{N}$ such that $u(-\frac{1}{n})<u(1)<u(2)$ for any $n\in J\subseteq \mathbb{N}$. Thus, $\liminf u(-\frac{1}{n})<u(2)$, so $u$ fails to be lower semicontinuous at 2.

\end{example}



\section{Sufficient conditions}

Let's see an interesting property satisfied by Alexandrov topologies, but not by the Upper nor by the Scott topologies.

\begin{lemma}\label{Lgiltza}
Let $\sqsubseteq$ and $\precsim$ two preorders on $X$  and $\tau_A^\sqsubseteq$ and $\tau_A^\precsim$ the corresponding Alexandrov topologies. If $\precsim$ refines $\sqsubseteq$ (i.e.  $\sqsubseteq \subseteq \precsim$), then  $\tau_A^\precsim\subseteq \tau_A^\sqsubseteq$.

\end{lemma}
\begin{proof}
The inclusion  $\tau_A^\precsim\subseteq \tau_A^\sqsubseteq$ holds true if and only if any convergent net $(x_i)_{i \in I}$ on $(X,\tau_A^\sqsubseteq)$ also converges on $(X,\tau_A^\precsim)$.
By reduction to the absurd, suppose there is a net  $(x_i)_{i \in I}$ that converges to $x$ on  $(X,\tau_A^\sqsubseteq)$ but that fails to converge on $(X,\tau_A^\precsim)$. Thus, there exists an open neighbourhood  $U\in \tau_A^\precsim$ with $x\in U$ such that  $(x_i)_{i \in I}$ is not cofinally in $U$. Since the open sets are the up-sets, that means that  there is a subnet $(x_j)_{j\in J\sqsubseteq I}$ of  $(x_i)_{i \in I}$ such that $\neg (x\precsim x_j)$. Therefore, it also holds true that  $\neg (x\sqsubseteq x_j)$ so,  we have that $x\in U_{\sqsubseteq }(x)\in \tau_A^\sqsubseteq$ as well as $x_j\notin U_{\sqsubseteq }(x)$ (for any $j\in J$). Thus, the subnet $(x_j)_{j\in J}$ fails to converge to $x$ on  $(X,\tau_A^\sqsubseteq)$, which contradicts the hyphothesis\footnote{Here, it is used that any subnet of a convergent net converges to the same point.}.
\end{proof}

This property is not satisfied in general by the Upper nor the Scott topologies. The following example is devoted to see that.

\begin{example}
Let $X$ be the 
union between $X_1=[0, +\infty)_1$ and $X_2=[0, +\infty)_2$. 
As in Example~\ref{Ex1}, we define the preorder $\sqsubseteq $ on $X$ by $x\sqsubseteq y $ if and only if $x,y\in X_k$ and $x\leq y$, with $k=1,2$. Hence, $x $ and $y$ are incomparable for any $x\in X_m$ and $y\in X_n$ with $n\neq m$.

Now we define a preorder $\precsim$ which refines the previous one as follows:

$x\precsim y \iff \left\{  \begin{array}{lcl}
 x\sqsubseteq y  &;& \, x, y\in X,\\
  x\in[0,1)_1 &;&\, y\in [1,+\infty)_2, \\
   x\in[0,1)_2 &;&\, y\in X_1,
\end{array}\right.\medskip$\\

Then, the lower set $L_\precsim (2_2)$ is closed on $\tau_u^\precsim$ (in other words, $X\setminus L_\precsim (2_2)$ is  open) whereas it is not in $\tau_u^\sqsubseteq$. Hence, $\tau_u^\sqsubseteq$ cannot be finer than $\tau_u^\precsim$. Notice too that $1_1\notin L_\precsim (2_2)$, but it is contained in the clousure $\overline{ L_\precsim (2_2)}$ with respect to $\tau_u^{\sqsubseteq}$, thus, $L_\precsim (2_2)$ is not closed on $\tau_u^\sqsubseteq$.

\begin{figure}[htbp]
\begin{center}
\begin{tikzpicture}[scale=0.8]
\draw[thick] (-3.5,0.2) node[anchor=east] {$\mathbb{R}_1$};
    \draw[] (2.5,0) -- (5,0);
\draw[dashed] (5,0) -- (6,0);
  \draw[] (2.5,-2) -- (5,-2); 
\draw[dashed] (5,-2) -- (6,-2);
\draw[thick] (2.9,-1.7) node[anchor=east] {\small $1_2$};
\draw[thick] (2.55,-2) node[anchor=east] { $\bullet$};
\draw[thick] (2.5,-2) node[anchor=east] {\small $[$};

\draw[thick] (0,0.3) node[anchor=east] {\small $0_1$};
\draw[thick] (-0.3,0) node[anchor=east] { $\bullet$};
\draw[thick] (-0.35,0) node[anchor=east] {\small $[$};

\draw[thick] (1.35,0) node[anchor=east] {\small $)$};
\draw[thick] (2.5,0) node[anchor=east] {\small $[$};

\draw[thick] (-2.7,-2) node[anchor=east] { $\bullet$};
\draw[thick] (-2.55,-1.7) node[anchor=east] {\small $0_2$};

\draw[thick] (-2.75,-2) node[anchor=east] {\small $[$};
\draw[thick] (-1.35,-2) node[anchor=east] {\small $)$};

\draw[dashed] (1.15,0) -- (2.3,-2);
\draw[dashed] (1,0) -- (2.5,0);

 \draw[] (-0.5,0) -- (1,0);
  \draw[] (-3,-2) -- (-1.5,-2); 

\draw[thick] (2.9,0.3) node[anchor=east] {\small $1_1$};
\draw[thick] (2.55,0) node[anchor=east] { $\bullet$};

\draw[dashed] (-1.5,-2) -- (-0.5,0);

\draw[thick] (-3.5,-1.8) node[anchor=east] {$\mathbb{R}_2$};

\draw[dashed] (-1.5,-2) -- (2.5,-2);
\end{tikzpicture}
\caption{Preorder defined on $[0,+\infty)_1\cup [0, +\infty)_2$.}
\end{center}
\end{figure}

\end{example}

Therefore, from  Lemma \ref{Lgiltza} the following corollary arises:

\begin{corollary}
Let $\precsim$ be a preorder defined on a topological space $(X,\tau)$. If $\tau$ is finer than the corresponding Alexandrov's topology $\tau_A^{\precsim}$, then any linear extension of the preorder is lower-semicontinuous on $\tau$.
\end{corollary}


In order to show that the Alexandrov's topology is not strange at all, we include the following result that shows that some authors  {\color{red} CITE} have already work on this spaces (at least in a subset of the corresponding set, maybe inconsciently) when they worked under the  assumption that any linear extension is lower semicontinuous.

\begin{theorem}
Let $\precsim$ be a preorder defined on a topological space $(X,\tau)$. Assume there is chain $(C, \precsim)$ included in $(X, \precsim)$ and an element $x\in X$ such that $x\bowtie c$ for any $c\in C$. If any linear extension of the preorder is lower-semicontinuous in $\tau$, then the reduction of $\tau$ to $C$ (that is, $\tau_{|C}$) is finer that the corresponding Alexandrov's topology $\tau_A^{\precsim}$ on $C$.
\end{theorem}
\begin{proof}
First, since $x\bowtie c$ for any $c\in C$, for a given $c_0\in C$ 
we can define an extension $\precsim_1$ of the preorder but now imposing that $c_0\prec_1 x$ and including the corresponding transitive clousure. Dually, we define the extension $\precsim_2$ of the preorder  imposing that $c_0\prec_2 x$.

By Szpilrajn extension theorem \cite{Spi},  there exists a linear extension $\leq_1^{c_0}$ such that  $c_0<_1^{c_0} x$ and $x<_1^{c_0} c'$ for any $c'\in C$ with $c_0\prec c'$. Anolagously, there exists another linear extension $\leq_2^{c_0}$ such that $x<_2^{c_0} c_0$ and $c'<_2 x$ for any $c'\in C$ with $c'\prec c_0$. That is, we can embed $x$ in any desired point $c$ of $C$, achieving two linear orders on $C\cup\{x\}$: $\leq_1^{c}$ and $\leq_2^{c}$.

Since, by hypothesis, any linear extension of the preorder is lower-semicontinuous in $\tau$, we deduce that the subsets $L_{\leq_1}(x)=L_{\leq_1}(c_0)\cup \{x\}$ and  $L_{\leq_2}(x)=L_{<_2}(c_0)$ are closed in $\tau$. Hence, restricting to $C$, it holds that both $L_{\precsim}(c_0)$ and  $L_{\prec}(c_0)$ are closed in $\tau_{|C}$, and that will hold for any $c_0\in C$. Thus, any down set is closed in $(C, \tau_C)$, so we conclude that $\tau_C$ is finer than the corresponding Alexandrov topology.
\end{proof}







\section{Further comments}
In the present paper the authors have focused on semicontinuous finite Richter-Peleg multi-utility. As we said in Section~\ref{s2}, there is a theorem of Evren and Ok \cite{EvO} that characterizes the existence of a semicontinuous finite multi-utility under the assumption of near-completness and second countability.

In a previous paper, some of the authors of the present work studied the idea of \emph{partial representability}. In order to illustrate some of these ideas, they introduced some examples. One of these examples was commented in a final remark as a possible counterexample for the aforementioned Theorem~\ref{Tok}, however, that was not correct at all. The mistake lies in the fact that the example fails to satisfy the hypothesis of   Theorem~\ref{Tok}, hence, it is not a counterexample. In particular, the lower contour set $L_{\precsim}(0'5)$ of the example fails to be closed.
Therefore, (and after a deeper study of the  proof) the authors believe that the mentioned theorem given by Evren and Ok is correct.




\section{Conclusions} \label{s6}

After this work, we conclude that if we want to search for lower  semicontinuous and finite Richter-Peleg multi-utilities, it is necessary to start the study from topological spaces that refine the corresponding Scott topology, and not --as usual-- from Upper topologies. We also show that under some hypothesis assumed in the literature (precisely, under the  assumption that any linear extension is lower semicontinuous) the topology of the space has a strong relation with the Alexandrov topology.




\begin{thebibliography}{}



\bibitem{Alc} Alcantud, J.C.R., G. Bosi, M. Zuanon,: Richter-Peleg multi-utility representations of preorders. {\em Theory and Decision} \textbf{80} \rm (2016) 443-450.

\bibitem{arenas}  F.G. Arenas, {Alexandroff spaces}, Acta Math. Univ. Comenianae, (1999).














 \bibitem{subm}  Bosi, G.,   Estevan, A.,  Guti\'errez Garc\'{\i}a, J.,  and Indur\'ain, E., Continuous representability of interval orders, the topological compatibility setting,  \emph{Internat. J. Uncertain. Fuzziness Knowledge-Based Systems} \textbf{23 (3)}\rm (2015) 345-365.

 \bibitem{partial}  Bosi, G.,   Estevan, A.,  Zuanon, M., Partial representations of orderings,  \emph{Internat. J. Uncertain. Fuzziness Knowledge-Based Systems}. \textbf{26 (03)}  (2018), 453-473.


 \bibitem{BosHer1}  Bosi, G.,  Herden, G.: On continuous multi-utility representations of semi-closed and closed preorders, {\em Mathematical Social Sciences} \textbf{79} \rm (2016) 20-29.






 \bibitem{BRME} Bridges, D.S.,   Mehta, G.B.: Representations
of Preference Orderings.
 Berlin-Heidelberg-New York: Springer-Verlag, 1995.






\bibitem{lat} Davey, B.A.; Priestley, H. A.. Introduction to Lattices and Order (Second ed.). Cambridge University Press, 2002. ISBN 0-521-78451-4.


 \bibitem{de}  Debreu, G.: Representation of a preference
ordering by a numerical function, In R. Thrall, C. Coombs and R.
Davies (Eds.). Decision processes, New York: John Wiley, 1954.




 \bibitem{Debr} Debreu, G.: Continuity Properties of Paretian Utility. {\em International Economic Review} \textbf{5} (1964) 285-293.







\bibitem{EvO}  Evren, O.,    Ok,  E.A.,  On the multi-utility representation of preference relations,
{\em Journal of Mathematical Economics} {\bf 47} (2011),  554-563.





\bibitem{G}  Gensemer, S.H.: Continuous semiorder representations. {\em Journal of Mathematical Economics}  \textbf{16} \rm (1987) 275-289.

\bibitem{lattices} 
 Gierz, G.,  Hofmann, K. H.,  Keimel,  K.,   Lawson, J. D.,  Mislove,   M.,  Scott,  D. S.,
Continuous Lattices and Domains (Encyclopedia of Mathematics and its Applications) 1st Edition. Cambridge Univer-
sity Press, 2003.


 
\bibitem{Kam}  Kaminski, B.,  On quasi-orderings and multi-objective functions, {\em European Journal of Operational Research} {\bf 177} (2007), 1591-1598.





\bibitem{Min2}  Minguzzi, E.,  Normally Preordered Spaces and Utilities, {\em Order} {\bf 30} (2013), 137-150.







\bibitem{Ok}  Ok,  E. A.,  Utility representation of an incomplete preference relation, {\em Journal of  Economic
Theory} {\bf 104} (2002), 429-449.

\bibitem{Sco} Scott, D.,   Suppes, P.: Foundational aspects of theories of measurement. {\em The Journal of Symbolic Logic} \textbf{23} \rm (1958) 113-128.

\bibitem{tim} Speer, T., 
A Short Study of Alexandroff Spaces, arXiv:0708.2136v1 (2007).

\bibitem{Spi} Szpilrajn, E., Sur l'extension de l'ordre partiel, Fundamenta Mathematicae, \textbf{16} (1930) 386–389.
\end{thebibliography}
\end{document}